\documentclass[12pt]{elsarticle}
\makeatletter
\def\ps@pprintTitle{%
   \let\@oddhead\@empty
   \let\@evenhead\@empty
   \let\@oddfoot\@empty
   \let\@evenfoot\@oddfoot
}
\makeatother
\usepackage{graphicx}
\usepackage{comment}
\usepackage{color}
\usepackage{pdfsync}
\usepackage{amsmath, amsthm, amscd, amsfonts, amssymb, graphicx, color, float}
\usepackage[bookmarksnumbered, colorlinks, plainpages]{hyperref}
\hypersetup{colorlinks=true,linkcolor=red, anchorcolor=green, citecolor=cyan, urlcolor=red, filecolor=magenta, pdftoolbar=true}
\usepackage{float}
\usepackage{mathrsfs}
\usepackage{tikz-cd}

\allowdisplaybreaks

\usepackage{amsmath,amsfonts,amscd,bezier,amsthm,amssymb}
\usepackage{parskip}

\makeatletter
\def\ps@pprintTitle{%
   \let\@oddhead\@empty
   \let\@evenhead\@empty
   \let\@oddfoot\@empty
   \let\@evenfoot\@oddfoot
}
\makeatother
\newtheorem{theorem}{Theorem}[section]
\newtheorem{lemma}[theorem]{Lemma}
\newtheorem{corollary}[theorem]{Corollary}
\newtheorem{definition}[theorem]{Definition}
\newtheorem{remark}[theorem]{Remark}

\newtheorem{proposition}[theorem]{Proposition}
\newcommand{\dl}{\delta_{\text{left}}}
\newcommand{\dR}{{_{RL}}\delta^{\alpha}}
\newcommand{\dC}{{_C}\delta^{\alpha}}

\newcommand{\R}{\mathbb{{R}}}

\newcommand{\N}{\mathbb{{N}}}

\newcommand{\C}{\mathbb{{C}}}

\numberwithin{equation}{section}
\allowdisplaybreaks
\makeatletter
\DeclareRobustCommand{\eqrefp}[2]{%
\textup{\tagform@{\refp{#1}{#2}}}}
\makeatother
\DeclareRobustCommand{\refp}[2]{%
\expandafter\ifx\csname r@#1\endcsname\relax
  \textbf{??}%
  \else
    \edef\areferencia{\ref{#1}-)}%
    \expandafter\eqrefpaux\areferencia-#2%
  \fi}
\def\eqrefpaux#1-#2){#1}

\begin{document}
\begin{frontmatter}
\title{Subordination principle, Wright functions and large-time behaviour for the discrete in time fractional diffusion equation}

\author{Luciano Abadias$^{\dagger}$}
\address{Departamento de Matem\'aticas, Instituto Universitario de Matem\'aticas y Aplicaciones, Universidad de Zaragoza, 50009 Zaragoza, Spain. \\ $^{\dagger}$labadias@unizar.es}
\author{Edgardo Alvarez$^{\ddagger}$ and Stiven D\'iaz$^{*}$ }
\address{Universidad del Norte, Departamento de Matem\'aticas y Estad\'istica, Barranquilla, Colombia. \\
$^{\ddagger}$ealvareze@uninorte.edu.co  \\$^{*}$stivend@uninorte.edu.co}


\begin{abstract}
The main goal in this paper is to study asymptotic behaviour  in $L^p(\R^N)$ for the solutions of the fractional version of the discrete in time $N$-dimensional diffusion equation, which involves the Caputo fractional $h$-difference operator. The techniques to prove the results are based in new subordination formulas involving the discrete in time Gaussian kernel, and which are defined via an analogue in discrete time setting of the scaled Wright functions. Moreover, we get an equivalent representation of that subordination formula by Fox H-functions.
\end{abstract}
\end{frontmatter}


\section{Introduction}

One of the most important aspects in the study of evolution problems is the asymptotic behaviour of solutions. In particular, since J. Fourier introduced the classical heat equation $u_t=\Delta u$ in 1822, see \cite{Fourier}, to model diffusion phenomena, several authors have invested their time and effort in researching the large time behaviour of diffusion processes. For example in \cite{DD,ZE,GV,KS,N} the authors studied large-time behaviour and other asymptotic estimates for diffusion problems in $\R^N,$ and in \cite{D2, GV2} for open bounded domains. Estimates for heat kernels on manifolds have been studied in \cite{Gr,L,D}, and in \cite{M} the author obtained Gaussian upper estimates for the heat kernel associated to the sub-laplacian on a Lie group.

In last years, everyone knows the relevance of fractional calculus on the topic of evolution problems and partial differential equations. Recently, from different points of view see \cite{Abadias-Alvarez-18, Kemppainen}, asymptotic estimates of solutions of fractional diffusion phenomena have been stated. In particular, in \cite{Abadias-Alvarez-18} the authors use subordination formulas as the main tool, however in \cite{Kemppainen} the authors work directly with estimates of the fractional fundamental solution. We propose to use both tools for our aims, as we will see throughout the paper.

On the other hand, finite differences were introduced some centuries ago, and they have been used in different mathematical problems, mainly in approximation of solutions of differential problems for the numerical solution of differential equations and partial differential equations. The most knowing ones are the forward, backward and central differences (the forward and backward differences are associated to the Euler, explicit and implicit, numerical methods). In the last years, several authors have been working in partial difference-differential equations (\cite{Abadia-Alvarez-20,ADT, Abadias-Lizama, C, C2, PAMS, LR}) from the point of mathematical analysis, more precisely, harmonic analysis, functional analysis and fractional differences. These last ones, fractional differences, are nowadays a topic of important research, see for example \cite{Abadia-Alvarez-20,Go-Peter,Lizama-(2015),PAMS,Mo-Wy,Ponce-(2020)} and references therein. Several aspects of such problems have been studied in that papers: maximal regularity, stability, fractional discrete resolvent operators, among others.

The main goal in this paper is to study asymptotic decay and large time behaviour on $L^p(\R^N)$ of solutions of the following fractional discrete in time heat problem, as the authors in \cite{Abadia-Alvarez-20} do in the classical case ($\alpha=1$). We consider \begin{equation}\label{1.2}
\left\{
\begin{array}{lll}
\dC u(nh,x) -\Delta u(nh,x)=0,\quad n\in\N,\,x\in\mathbb{R}^N, \\ \\
u(0,x) =f(x),\\
\end{array}
\right.
\end{equation}
where $0 <\alpha \leq 1$, $\dC$ is the Caputo fractional $h$-difference operator (Section \ref{Section 3}), $\Delta$ denotes the Laplace operator acting in space, $u$ is defined on $\N^{h}_0\times\R^{N}$ and $f$ is a function defined on $\R^{N}$. We use subordination formulas to write the solution of the previous problem. Indeed we will write the solution $$u(nh,x)=\sum_{j=1}^{\infty}\varphi^h_{\alpha,1-\alpha}(n-1,j-1)(\mathcal{G}_{j,h}*f)(x),$$ where $\mathcal{G}_{j,h}$ is the discrete Gaussian kernel associated to the discrete in time heat problem given in \cite{Abadia-Alvarez-20}, and $\varphi^h_{\alpha,\beta}$ is discrete scaled Wright function (which is introduced in Section \ref{dsftsf}):
\begin{equation*}
\varphi^h_{\alpha,\beta}(n,j):=
              \frac{1}{2 \pi i} \displaystyle\int_{\Upsilon} \frac{1}{z^{n+1}}\frac{\left(1-h(\frac{1-z}{h})^{\alpha} \right)^{j}}{(\frac{1-z}{h})^{\beta}}\,dz,  \quad  n,j\in\mathbb{N}_0,\,\beta\geq0.
\end{equation*}
Note that it is a generalization of one given in \cite{Alvarez-Diaz-Lizama-2020}.

Previous subordination formula and the known results about the classical case ($\alpha=1,$ see \cite{Abadia-Alvarez-20}) are the key tools to study the decay and the asymptotic behaviour on $L^p(\R^n)$ for the solutions of \eqref{1.2}.

The paper is organized as follows. In second section we revisited known useful results for our aims. We recall the concept of Wright functions, which plays a key role in the subordination formulae for resolvent families in the continuous case, and whose properties help us to study important results for the discrete setting. Also, we present basic properties for the discrete Gaussian kernel in $\R^N$ (solution of \eqref{1.2} for $\alpha=1$).

Section 3 is devoted to state our fractional discrete setting. We consider the classical backward difference on the mesh of step $h>0.$ This difference allows to consider adequate notions of discrete fractional sum, and discrete fractional difference in the Riemann-Liouville and Caputo sense, for our purpose. Several useful properties will be shown.

In Section 4 we introduced the Mittag-Leffler sequences which are the solutions of \eqref{1.2} in the scalar case. Such Mittag-Leffler functions motivate the study of the Wright functions in the discrete setting, which are one of the key tools in this paper. Proposition \ref{Prop-Levy} contains many properties of both type functions, and the relations between them. Such result can be considered the analogue one to \cite[Theorem 3]{Abadia-Miana} in the continuous case.

In Section 5 we focus on the study of the fundamental solution of $\eqref{1.2}.$ We prove by the subordination formula given by the discrete heat kernel and Wright functions that, in fact, the integral defined by such subordination formula is the solution. One of the flashy facts of the subordination formula proposed (see \eqref{subordination1}) is its relation with the Gaussian kernel and the Fox H-function, as Proposition \ref{GaussianSemigroup-FS} and Proposition  \ref{FoxH-FS} show. Also we state basic properties of the fundamental solution. One of them is that the integral over $\R^N$ of the fundamental solution is 1, which gives directly the mass conservation principle for solutions of \eqref{1.2}.

Section 6 contains the asymptotic $L^p$-results. We state the $L^p$-decay for the fundamental solution and its gradient. These decay bounds allow to get the $L^p$-decay for the solution of $\eqref{1.2}$ (Theorem \ref{teorema4.1}), and also the large time behaviour. In this asymptotic behaviour is reflected the mass conservation principle because the solution converges to the total mass times the fundamental solution (see Theorem \ref{Theorem5.1}).

\section{Preliminaries}
\subsection{Continuous fractional calculus.}  In this part, we recall some concepts and basic results about fractional calculus in continuous time.
Let $0<\alpha<1$ and $f$ be a locally integrable function. The Riemann-Liouville fractional derivative of $f$ of order $\alpha$ is given by
$$_{R}D^{\alpha}_t f(t) := \frac{d}{dt}\displaystyle\int_0^t \dfrac{(t-s)^{-\alpha}}{\Gamma(1-\alpha)}f(s) \, ds, \ t \geq 0.$$
The Caputo fractional
derivative of  order $\alpha$ of a function $f$ is defined by
\begin{align*}
_{C}D^{\alpha}_t f(t)&:=\dfrac{1}{\Gamma(1-\alpha)} \displaystyle\int_0^t  (t-s)^{-\alpha }f'(s) \, ds, \ t \geq 0,
\end{align*}
where $f'$ is the first order distributional derivative of $f(\cdot)$, for example if we assume that $f(\cdot)$ has locally integrable distributional derivative up to order one.  Then, when $\alpha=1$ we obtain $_{C}D^{\alpha}_t:=\dfrac{d}{dt}$.  For more details, see for example  \cite{Mainardi, Miller}.

The Mittag-Leffler functions are given by
\begin{align*}
E_{\alpha,\beta}(z):=\displaystyle\sum_{n=0}^{\infty}\frac{z^n}{\Gamma(\alpha n+\beta)}, \qquad \alpha,\,\beta>0,\,z\in\C.
\end{align*}
We write $E_{\alpha}(z):=E_{\alpha,1}(z).$ They are solutions of the fractional differential problems
$$_C D_t^{\alpha}E_{\alpha}(\omega t^{\alpha})=\omega E_{\alpha}(\omega t^{\alpha}),$$
and
$$_R D_t^{\alpha}\biggl(t^{\alpha-1}E_{\alpha,\alpha}(\omega t^{\alpha})\biggr)=\omega t^{\alpha-1}E_{\alpha,\alpha}(\omega t^{\alpha}),$$ for $0<\alpha<1,$ under certain initial conditions. 
Their Laplace transform is
\begin{equation*}
\int_0^{\infty}e^{-\lambda t}t^{\beta-1}E_{\alpha,\beta}(\omega t^{\alpha})\,dt=\frac{\lambda^{\alpha-\beta}}{\lambda^{\alpha}-\omega}, \qquad \mbox{Re}(\lambda)>\omega^{\frac{1}{\alpha}},\,\omega>0.
\end{equation*}
For more details about the Mittag-Leffler function $E_{\alpha,\beta}$ see \cite[Chapter 18]{EMOTB}.

Recall the definition of the Wright type function (see \cite{Mainardi})
\begin{equation*}\label{wright}
W_{\lambda,\mu}(z)=\displaystyle\sum_{n=0}^{\infty}\dfrac{z^n}{n!\Gamma(\lambda n+\mu)}= \frac{1}{2\pi i}
\int_{H_a} \sigma^{-\mu} e^{\sigma +z \sigma^{-\lambda}} d\sigma, \quad
 \lambda>-1,\,\mu\geq 0,\,z\in \mathbb{C},
\end{equation*}
where $H_a$ denotes the Hankel path defined as a contour which starts and ends
at $-\infty$ and encircles the origin once counterclockwise.

For $0<\alpha<1$ and $\beta\geq 0,$ the scaled Wright function in two variables  $\psi_{\alpha,\beta}$ (and which was introduced by Abadias and Miana in \cite{Abadia-Miana})  is given by
\begin{align}\label{ScaledWright}
\psi_{\alpha,\beta}(t,s):=t^{\beta-1}W_{-\alpha,\beta}(-st^{-\alpha}),\qquad t>0,\ s\in\C.
\end{align}

Note that using the change of variable $z=\frac{\sigma}{t},$ we get the integral representation $$\psi_{\alpha,\beta}(t,s)=\frac{1}{2\pi i}\int_{Ha}z^{-\beta}e^{tz-sz^{\alpha}}\,dz, \qquad t,s>0.$$

Many properties about such functions that we will use along the paper appear in \cite{Abadia-Miana}.\\ \\
Next, we recall the definition of Fox H-funtions. Let $m, n, p, q\in\N_0$ such that $0\leq m \leq q$, $0\leq n \leq p$. Let  $a_i,b_j\in \C$  and $\alpha_i,\beta_j\in \R_+.$ The Fox H-function is defined via a Mellin-Barnes type integral,
\begin{align*}
H^{mn}_{pq}(z)&:=H^{mn}_{p,q}\left[\begin{array}{c}
z \left\vert \begin{array}{c}
 (a_{i}, \alpha_{i})_{1,p} \\
 (b_{j}, \beta_{j})_{1,q}
\end{array}
\right.
\end{array}
\right] =\frac{1}{2\pi i} \int_{\gamma}\mathcal{H}^{mn}_{pq}(s) z^{-s} \ ds,
\end{align*}
where
\begin{align*}
(a_{i}, \alpha_{i})_{1,p}:=(a_{1}, \alpha_{1}), &\cdots, (a_{p}, \alpha_{p}),  \\
  (b_{j}, \beta_{j})_{1,q}:=(b_{1}, \beta_{1}), &\cdots, (b_{q}, \beta_{q}),
\end{align*}

\begin{align*}
\mathcal{H}^{mn}_{pq}(s)=\dfrac{\displaystyle\prod_{j=1}^{m}\Gamma(b_j+\beta_j s)\displaystyle\prod_{i=1}^{n}\Gamma(1-a_i-\alpha_i s)}{\displaystyle\prod_{i=n+1}^{p}\Gamma(a_i+\alpha_i s)\displaystyle\prod_{j=m+1}^{q}\Gamma(1-b_j-\beta_j s)},
\end{align*}
and $\gamma$ is the infinite contour in the complex plane which separates the poles
\begin{align*}
b_{jl}=\frac{-b_j-l}{\beta_j} \ (j=1,\cdots, m; \ l\in\N_{0})
\end{align*}
of the Gamma function $\Gamma(b_j+\beta_{j}s)$ to the left of $\gamma$ and the poles
\begin{align*}
a_{ik}=\frac{1-a_i+k}{\alpha_i} \ (i=1,\cdots, n; \ k\in\N_{0})
\end{align*}
to the right of $\gamma$.  For more details of this type of functions, see \cite{H-transforms}. \\ \\
Finally, we recall that the Gaussian kernel $G_t(x)$ is defined by
\begin{align}\label{gaussian-semigroup}
G_t(x)=\frac{1}{(4\pi t)^{1/2}}e^{-\frac{|x|^2}{4t}}, \qquad t>0, \ x\in\R^{N}.
\end{align}

Let
$$\widehat{u}(\xi)=\mathcal{F}(u)(\xi)=(2\pi)^{-N/2}\int_{\mathbb{R}^N}e^{-ix\cdot\xi}f(x)dx$$
and
$$\mathcal{F}^{-1}(u)(\xi):=\mathcal{F}(u)(-\xi)$$
denote the Fourier and inverse Fourier transform of $u$, respectively.

The function $G_t(x)$ has the following properties (see for example
\cite{Evans}).\\


\begin{proposition}\label{gauss-prop}
The Gaussian kernel satisfies:
\begin{enumerate}[$(i)$]
\item $\displaystyle G_t(x)>0$.
\item $\displaystyle \int_{\R^N}G_t(x)\,dx=1$.
\item $\displaystyle\mathcal{F}(G_t)(\xi)=e ^{-t\vert\xi \vert^2},$ \quad $\xi\in\R^N.$
\item $\displaystyle \int_{\R^N}|x|^2 G_{t}(x)\,dx=2Nt.$
\end{enumerate}
\end{proposition}

\subsection{Discrete diffusion equation.}
In \cite{Abadia-Alvarez-20}, for $h>0$ the authors defined the heat kernel in discrete time on the mesh of step $h$ as
\begin{equation}\label{discretegaussian}
\mathcal{G}_{n,h}(x):= \frac{1}{h^n\Gamma(n)} \int_0^\infty e^{-t/h} t^{n-1} G_t(x)\,dt,\quad  n\in\N,x\in\R\setminus\{0\}.
\end{equation}
Moreover, they proved the following proposition.
\begin{proposition}\label{DG-Properties}
The function $\mathcal{G}_{n,h}$ satisfies:
	\begin{enumerate}[$(i)$]
	\item  $\displaystyle\mathcal{G}_{n,h}(x)>0,\quad n\in\N,x \in \R.$
	\item  $\displaystyle \int_{\R}^{}\mathcal{G}_{n,h}(x)\ dx=1$.
	\item  $\displaystyle\mathcal{F}({\mathcal{G}}_{n,h})(\xi)=\frac{1}{(1+h|\xi|^2)^n},\quad \xi\in\R.$
\item  $\dfrac{\mathcal{G}_{n,h}(x)-\mathcal{G}_{n-1,h}(x)}{h} = \Delta \mathcal{G}_{n,h}(x),\quad n\geq 2,\ x\in\R.$
	\end{enumerate}
\end{proposition}

Next, given a function $f$ defined on $\R^N$ and $h>0$, Abadias and Alvarez proved (see \cite{Abadia-Alvarez-20}) that  $w(nh,x):=(\mathcal{G}_{n,h}\ast f)(x)$ is the unique solution of the problem
\begin{equation}\label{SolFirtOrderEquation}
\left\{
\begin{array}{lll}
\dfrac{w(nh,x)-w((n-1)h,x)}{h}= \Delta w(nh,x) ,\quad n\in\N,\,x\in\mathbb{R}^N\setminus\{0\}, \\ \\
w(0,x) =f(x).
\end{array}
\right.
\end{equation}

\section{Discrete fractional calculus}\label{Section 3}
In this section we recall the definition of Ces\`aro numbers and some useful properties of them. Also, we introduce the discrete time setting where we will work, and the corresponding associated fractional calculus.

For an arbitrary $\alpha\in\C,$ we denote by $k^{\alpha}(n)$ the Ces\`aro numbers which are the Fourier coefficients of the holomorphic function on the disc $(1-z)^{-\alpha},$ that is,
\begin{align}\label{ZT-K}
\frac{1}{(1-z)^{\alpha}}=\sum_{n=0}^{\infty}k^{\alpha}(n)z^n.
\end{align}
It is known that the expression of the  Ces\`aro numbers is given by
\begin{equation}\label{kernel_k_1}
 k^{\alpha}(n):=\left\{ \begin{array}{lcc}
             \dfrac{\alpha(\alpha+1)\cdots \cdots ( \alpha+n-1)}{n!}, &   \qquad   n\in\N, \\ \\
 \qquad  \qquad \qquad 1, & \qquad    n= 0.
             \end{array}
   \right.
\end{equation}
Note that $k^0(n):= \delta_{0}(n)$ is the Kronecker delta. Sometimes, to make more easily computations with the Ces\`aro sequence $k^{\alpha}$ for $\alpha\in\C \setminus \{0,-1,-2,...\}$, we use an equivalent expression of \eqref{kernel_k_1}, namely
\begin{equation}\label{kernel_k}
k^{\alpha}(n)= \dfrac{\Gamma(n+\alpha)}{\Gamma(\alpha)\Gamma(n+1)},
 \end{equation}
where $\Gamma(\cdot)$ is the gamma function.

We recall the following properties of  $k^{\alpha}$ which appear for example  in  \cite{GoLi19,PAMS,Zygmund}.
\begin{proposition}\label{prop-Cesaro} The following properties hold:
\begin{enumerate}[$(i)$]
\item For $\alpha>0$, $k^{\alpha}(n)>0$, $n\in\N_0$.
\item For all $\alpha,\beta\in\C$, we have the semigroup property
\begin{align}\label{semigroup-k}
 \sum_{j=0}^{n} k^{\alpha}(n-j) k^{\beta} (j)=k^{\alpha+\beta}(n).
\end{align}
\item For $\alpha>0,$
\begin{align}\label{asin-kernel}
k^{\alpha}(n)=\frac{n^{\alpha-1}}{\Gamma(\alpha)}\left( 1 + \mathcal{O}\left(\frac{1}{n}\right)\right), \quad n\in\N.
\end{align}
\item  For $\alpha>0$
\begin{align}\label{inde-k}
k^{\alpha}(n+1)=\dfrac{\alpha+n}{n+1}k^{\alpha}(n).
\end{align}
\end{enumerate}
\end{proposition}

Let $h>0$ and $f$ be a sequence defined on $\N_0^h:=\{0,h,2h,\ldots\},$ the backward difference of the sequence $f$ is defined by
\begin{center}
$\dl f(nh):=\dfrac{f(nh)-f((n-1)h)}{h},$  \quad $n\in \N.$
\end{center}

Taking into account the previous  definition, we get  the following  result.

\begin{proposition}\label{delta-k}
Let $0<\alpha<1$ and $\rho_{\alpha}(nh):=h^{\alpha}k^\alpha(n)$. Then,
$$\delta_{\text{left}} \rho_{\alpha}(nh)=h^{\alpha-1}k^{\alpha-1}(n), \ n\in\N.$$
\end{proposition}
\begin{proof}
By \eqref{inde-k} and the property of  Gamma function $z\Gamma(z)=\Gamma(z+1)$ the result follows.
\end{proof}

The definitions of the fractional sum and fractional difference operators (in the sense of Riemann-Liouville and Caputo) with the sequence $ k^\alpha$ were initially proposed by C. Lizama in \cite{Lizama-(2015)}. Recently, R. Ponce in \cite{Ponce-(2020)} defines a generalization. We make a slightly  modification (since our index starts at one) to this definition.\\

\begin{definition} Let $f$ be a sequence defined on $\N_0^h$. For $\alpha \geq 0$, the $\alpha$-th fractional sum of $f$ is defined by means of the formula
$$\delta^{-\alpha}f(nh):=h^{\alpha}\displaystyle\sum_{j=1}^n k^{\alpha}(n-j)f(jh),\quad n\in\N.$$
Note that, for $\alpha=0$, $\delta^{-\alpha}f(nh)=f(nh).$
\end{definition}

As a direct consequence of the previous definition, the operator $\delta^{-\alpha}$ satisfies the semigroup property: 
\begin{align*}
\delta^{-\alpha}\delta^{-\beta}f(nh)=\delta^{-(\alpha+\beta)}f(nh), \ n\in\N.
\end{align*}

\begin{definition}\label{DRL} Let $0<\alpha<1$ and $f$  be a sequence defined on $\N_0^h$. The $\alpha$-th fractional $h$-difference in the sense of Riemann-Liouville of  $f$ is defined by
\begin{align}\label{RL}
\dR f(nh):=\dl \delta^{-(1-\alpha)}f(nh),\qquad n\in\N.
\end{align}
\end{definition}

\begin{proposition}\label{SumAndDeltaRL} For $0<\alpha<1$ the relation
$$\delta^{-\alpha} \dR f(nh)=f(nh), \ n\in\N,$$
holds.
\end{proposition}
\begin{proof}
First of all, we adopt the notation: $\sum^{0}_{j=1} f(jh)=0$.  Then, by \eqref{semigroup-k}, we get
\begin{align*}
\delta^{-\alpha} \dR f(nh)&=h^\alpha \sum^{n}_{j=1} k^{\alpha}(n-j)\dR f(jh)\\
&=h^\alpha \sum^{n}_{j=1} k^{\alpha}(n-j)\dl  \delta^{-(1-\alpha)}f (jh) \\
&=h^{\alpha-1} \sum^{n}_{j=1} k^{\alpha}(n-j) \delta^{-(1-\alpha)}f (jh)-h^{\alpha-1} \sum^{n}_{j=1} k^{\alpha}(n-j) \delta^{-(1-\alpha)}f ((j-1)h)\\
&= \sum^{n}_{j=1} k^{\alpha}(n-j)\sum^{j}_{i=1}k^{1-\alpha}(j-i) f (ih)-\sum^{n}_{j=1} k^{\alpha}(n-j) \sum^{j-1}_{i=1}k^{1-\alpha}(j-1-i)f (ih)\\
&= \sum^{n}_{j=1} k^{\alpha}(n-j)\sum^{j}_{i=1}k^{1-\alpha}(j-i) f (ih)-\sum^{n}_{j=2} k^{\alpha}(n-j) \sum^{j-1}_{i=1}k^{1-\alpha}(j-1-i)f (ih)\\
&= \sum^{n-1}_{j=0} k^{\alpha}(n-1-j)\sum^{j}_{i=0}k^{1-\alpha}(j-i) f ((i+1)h)\\\
&\qquad \qquad  -\sum^{n-2}_{j=0} k^{\alpha}(n-2-j) \sum^{j}_{i=0}k^{1-\alpha}(j-i)f ((i+1)h)\\
&= \sum^{n-1}_{j=0}  f ((i+1)h) - \sum^{n-2}_{i=0}f ((i+1)h)=f(nh).
\end{align*}
Consequently, we have $\delta^{-\alpha} \dR f(nh)=f(nh)$ for all $n\in \N.$
\end{proof}

\begin{definition}\label{Caputo} Let $0<\alpha<1$ and $f$ be a sequence defined on $\N_0^h$. The Caputo fractional difference of order $\alpha$ is defined by
\begin{equation}\label{Caputo}
\dC f(nh):=\delta^{-(1-\alpha)}\dl f(nh),\quad n\in\N.
\end{equation}
\end{definition}
Note that the previous definition gives $\dC=\dl$ for $\alpha=1.$

The next result shows the relation between Caputo and Riemann fractional difference.
\begin{proposition}\label{relation-CR} Let $0<\alpha< 1$. Then the following
identity holds
$$\dC f(nh)=\dR(f(nh)-f(0)),\qquad n\in\N.$$
\end{proposition}
\begin{proof}
By  \eqref{Caputo} and \eqref{RL},
\begin{eqnarray*}
\dC f(nh)&=&h^{1-\alpha}\sum^{n}_{j=1}k^{1-\alpha}(n-j)\dl f(jh) \\
&=&h^{-\alpha}\sum^{n}_{j=1}k^{1-\alpha}(n-j)f(jh) -h^{-\alpha}\sum^{n}_{j=1}k^{1-\alpha}(n-j)f((j-1)h) \\
&=&h^{-\alpha}\sum^{n}_{j=0}k^{1-\alpha}(n-j)f(jh) -h^{-\alpha}\sum^{n-1}_{j=0}k^{1-\alpha}(n-1-j)f(jh)-h^{-\alpha}k^{1-\alpha}(n)f(0) \\
&=&h^{-\alpha}\sum^{n}_{j=0}k^{1-\alpha}(n-j)f(jh) -h^{-\alpha}\sum^{n-1}_{j=0}k^{1-\alpha}(n-1-j)f(jh))\\
& &\hspace*{3cm} -h^{-\alpha}f(0)\left(\sum_{j=0}^{n} k^{1-\alpha}(n-j)-\sum_{j=0}^{n-1} k^{1-\alpha}(n-1-j) \right) \\
&=&h^{-\alpha}\sum^{n}_{j=0}k^{1-\alpha}(n-j)(f(jh)-f(0)) -h^{-\alpha}\sum^{n-1}_{j=0}k^{1-\alpha}(n-1-j)(f(jh)-f(0)).
\end{eqnarray*}
\end{proof}


\begin{corollary} Let $0<\alpha<1$ and $f$ be  a sequence defined on $\N^{h}_0$. Then the following identity holds
\begin{align*}
\delta^{-\alpha} \dC f(nh)=f(nh)-f(0).
\end{align*}
\end{corollary}
\begin{proof}
The result is immediate by Propositions \ref{SumAndDeltaRL} and \ref{relation-CR}.
\end{proof}

The Proposition \ref{relation-CR}  allows to establish the following property between the fractional difference operators.

\begin{proposition}
For all $0<\alpha<1$ and $0  \leq \beta   $ the relation holds
\begin{align*}
\delta^{-(\beta+1)} {_{RL}\delta}^{1-\alpha}f(nh))= \delta^{-(\beta+\alpha)}   f(nh), \qquad n\in\N.
\end{align*}
\end{proposition}
\begin{proof}
By Proposition \ref{relation-CR}, we have
\begin{align*}
\delta^{-(\beta+1)} {_{RL}\delta}^{1-\alpha}f(nh))&=h^{\beta+1}\sum^{n}_{j=1}k^{\beta+1}(n-j){_{RL}\delta}^{1-\alpha}f(jh) \\
&=h^{\beta+1}\sum^{n}_{j=1}k^{\beta+1}(n-j){_{C}\delta}^{1-\alpha}f(jh) +h^{\beta+\alpha}\sum^{n}_{j=1}k^{\beta+1}(n-j)  k^{\alpha}(j-1)f(0) \\
&=h^{\beta+1}\sum^{n}_{j=1}k^{\beta+1}(n-j){_{C}\delta}^{1-\alpha}f(jh) +h^{\beta+\alpha}\sum^{n-1}_{j=0}k^{\beta+1}(n-1-j)  k^{\alpha}(j)f(0) \\
&= \delta^{-(\beta+\alpha)} \delta^{-1} \delta_{\text{left}}  f(nh) +h^{\alpha+\beta} \sum^{n-1}_{j=0} k^{\alpha+\beta}(n-1-j)f(0) \\
&= \delta^{-(\beta+\alpha)}   f(nh)-  \delta^{-(\beta+\alpha)} f(0)+h^{\alpha+\beta} \sum^{n}_{j=1} k^{\alpha+\beta}(n-j)f(0) \\
&= \delta^{-(\beta+\alpha)}   f(nh).
\end{align*}
\end{proof}

\section{Special functions and subordination formulae}\label{dsftsf}

In this section we introduce a discrete version of the Mittag-Leffler and scaled Wright functions, which generalize the ones defined in \cite{Alvarez-Diaz-Lizama-2020}. Also, we present some interesting properties which will be useful along the paper.

Let $\alpha,\beta,h>0$ and $\lambda\in\C$. The  Mittag-Leffler sequences are given by
\begin{align}\label{MT-discrete}
\mathcal{E}^{h}_{\alpha,\beta}(\lambda,n):=\frac{h^{\beta}}{(n-1)!}\sum_{j=0}^{\infty}\frac{\Gamma(\alpha j+\beta+n-1)}{\Gamma(\alpha j+\beta)}(h^{\alpha}\lambda)^{j}, \quad n\in\N, \quad |\lambda|<\frac{1}{h^{\alpha}}.
\end{align}
The convergence of previous series can be justified by \eqref{asin-kernel}. Using the  Ces\`aro numbers \eqref{kernel_k}, one can rewrite \eqref{MT-discrete} as
$$\mathcal{E}^{h}_{\alpha,\beta}(\lambda,n)=\sum_{j=0}^{\infty}h^{\alpha j+\beta}k^{\alpha j + \beta}(n-1)\lambda^{j},  \quad   n\in\N, \quad  |\lambda|<\frac{1}{h^{\alpha}}.$$
Particularly, note that
$$
\mathcal{E}^{h}_{1,1}(\lambda,n)=\sum_{j=0}^{\infty}(h\lambda)^{j}k^{ j+1 }(n-1)=\sum_{j=0}^{\infty}(h\lambda)^{j}k^{ n }(j)=\frac{1}{h^{n}}(1/h-\lambda)^{-n},\quad |\lambda|<\frac{1}{h}.$$


\begin{proposition}
Let $0<\alpha<1$, $0<\beta,h$ and $\lambda\in\C$ such that $|\lambda|<\frac{1}{h^{\alpha}}$. The sequence
\begin{align*}
\varepsilon(nh):=\left\{ \begin{array}{lcc}
              \mathcal{E}^{h}_{\alpha,1}(\lambda,n),  &\qquad   n\in\N, \\ \\
 1,  &\qquad    n= 0
             \end{array}
   \right.
\end{align*}
is solution of the fractional difference problem
\begin{align*}
_{C}\delta^{\alpha}\varepsilon(nh)=\lambda \varepsilon(nh), \qquad n\in\N.
\end{align*}
\end{proposition}
\begin{proof} For $n\in\N$, we have
\begin{align*}
\delta^{1-\alpha} \varepsilon(nh)&=h^{1-\alpha}\sum^{n}_{j=1}k^{1-\alpha}(n-j) \mathcal{E}^{h}_{\alpha,1}(\lambda,j) \\
&=h^{1-\alpha}\sum^{n}_{j=1}k^{1-\alpha}(n-j) \sum_{w=0}^{\infty}h^{\alpha w+1}k^{\alpha w + 1}(j-1)\lambda^{w}\\
&=\sum_{w=0}^{\infty}h^{2-\alpha+\alpha w}\lambda^{w} \sum^{n}_{j=1}  k^{1-\alpha}(n-j)k^{\alpha w + 1}(j-1)\\
&=\sum_{w=0}^{\infty}h^{2-\alpha+\alpha w}\lambda^{w}\sum^{n-1}_{j=0}  k^{1-\alpha}(n-1-j)k^{\alpha w + 1}(j)\\
&=\sum_{w=0}^{\infty}h^{2-\alpha+\alpha w}  k^{2-\alpha+\alpha w}(n-1)\lambda^{w}\\
&=\sum_{w=1}^{\infty}h^{2-\alpha+\alpha w}  k^{2-\alpha+\alpha w}(n-1)\lambda^{w}+h^{2-\alpha}  k^{2-\alpha}(n-1),
\end{align*}
where have used \eqref{semigroup-k}. Now, by Proposition \ref{delta-k}, we get
\begin{align*}
\delta_{\text{left}}h^{2-\alpha}  k^{2-\alpha}(n-1)=h^{1-\alpha}  k^{1-\alpha}(n-1)
\end{align*}
and
\begin{align*}
\delta_{\text{left}} h^{2-\alpha+\alpha w}  k^{2-\alpha+\alpha w}(n-1)=h^{1-\alpha+\alpha w}  k^{1-\alpha+\alpha w}(n-1).
\end{align*}
Then,
\begin{align*}
_{R}\delta^{\alpha} \varepsilon(nh)=\sum_{w=0}^{\infty}h^{1+\alpha w}  k^{1+\alpha w}(n-1)\lambda^{w}+h^{1-\alpha}  k^{1-\alpha}(n-1)u(0).
\end{align*}
Hence, by Proposition \ref{relation-CR} we get the result.
\end{proof}

Next, let us define the discrete scaled Wright function.
\begin{definition}\label{Def-Wright-discreta}  Let $0 <\alpha< 1$ and  $0 \leq \beta$ be given. For $n\in \N_{0}$ and  $h>0$, the discrete scaled Wright function $\varphi^{h}_{\alpha,\beta}$ is defined by
\begin{equation}\label{Wright-discreta}
\varphi^h_{\alpha,\beta}(n,j):=
              \frac{1}{2 \pi i} \displaystyle\int_{\Upsilon} \frac{1}{z^{n+1}}\frac{\left(1-h(\frac{1-z}{h})^{\alpha} \right)^{j}}{(\frac{1-z}{h})^{\beta}}\,dz,  \quad  j\in\mathbb{N}_0,
\end{equation}
where $\Upsilon$ is the path  oriented counterclockwise given by the circle centered at the origin and radius $0<r<1.$ Note that if $|z|<1,$ then $\frac{1-z}{h}$ belongs to the disc centered at $1/h$ and radius $1/h.$ So, for each $j\in\N_0,$ the function $z\to \frac{\left(1-h(\frac{1-z}{h})^{\alpha} \right)^{j}}{(\frac{1-z}{h})^{\beta}}$ is holomorphic on the unit disc. Therefore, by the Cauchy formula for the derivatives, we have defined $\varphi^h_{\alpha,\beta}(n,j)$ as the $n$-coefficient of the power series centered at the origin of such holomorphic function.
\end{definition}

In the following proposition, we present some useful properties of the discrete scaled Wright function $\varphi^h_{\alpha,\beta}$. Many of them follow the spirit of the analogue ones in the continuous case, see \cite[Theorem 3]{Abadia-Miana}.\\


\begin{proposition}\label{Prop-Levy}  Let $0 <\alpha < 1$, $0\leq \beta ,$ $0<h$ and $n,j\in\N_0$. The following properties hold:
\begin{enumerate}[$(i)$]
\item  $\varphi^{h}_{\alpha,\beta}(n,j)= h^{\beta}\displaystyle\sum^{j}_{i=0} \binom{j}{i}(-1)^{i}h^{i-\alpha i}k^{\beta-\alpha i}(n).$
\item $\varphi^h_{\alpha,\beta+\gamma}(n,j)=h^{\beta} \displaystyle\sum_{i=0}^{n} k^{\beta }(n-i) \varphi^{h}_{\alpha,\gamma}(i,j), \quad \gamma>0.$
\item $ \dfrac{1}{h^{n}\Gamma(n)}\displaystyle\int_0^{\infty}e^{-s/h}s^{n-1} \psi_{\alpha,\beta}(s,t) \ ds= e^{-t/h}\displaystyle\sum_{j=1}^{\infty} \varphi^h_{\alpha,\beta}(n-1,j-1)\frac{t^{j-1}}{ h^{j}(j-1)!},\  t>0 ,$\\ \\
where $\psi_{\alpha,\beta}$ is given by \eqref{ScaledWright}.
\item  $\displaystyle\sum^{\infty}_{j=1}\varphi^{h}_{\alpha,\beta}(n-1,j-1)\frac{1}{h^{j}}\left(\frac{1}{h}-\lambda\right)^{-j}=\frac{1}{h}\mathcal{E}^h_{\alpha,\alpha+\beta}(\lambda,n)$, \ $n\in\N, \ \vert \lambda\vert<\dfrac{1}{h^{\alpha}}$.
\item $\varphi^{h}_{\alpha,\beta}(n,j)-\varphi^{h}_{\alpha,\beta}(n,j+1) =h\varphi^{h}_{\alpha,\beta-\alpha}(n,j).$
\item $\varphi^h_{\alpha,0}(n,j+1)=\displaystyle\sum_{p=0}^{n}\varphi^h_{\alpha,0}(n-p,j)\varphi^h_{\alpha,0}(p,1)$.
\item   $\varphi^{h}_{\alpha,\beta}(n,j)\geq 0, \quad 0<h\leq 1.$
\item  $\displaystyle\sum_{i=0}^{\infty} \varphi^h_{\alpha,0}(i,j)=1$.
\item $\displaystyle\sum_{j=0}^{\infty} \varphi^h_{\alpha, \beta}(n,j)k^{\gamma}(j)=h^{\beta+\gamma(\alpha-1)}k^{\beta+\gamma\alpha}(n)$.

\end{enumerate}
\end{proposition}

\begin{proof}~
\begin{enumerate}[$(i)$]
\item  Note that for $|z|<1$ we can write $$
h^{\beta}\displaystyle\sum^{j}_{i=0} \binom{j}{i}(-1)^{i}h^{i-\alpha i}(1-z)^{\alpha i-\beta}=\frac{1}{(\frac{1-z}{h})^{\beta}}\displaystyle\sum^{j}_{i=0} \binom{j}{i}(-1)^{i}h^{i}\left(\frac{1-z}{h}\right)^{\alpha i}=\frac{\left(1-h(\frac{1-z}{h})^{\alpha} \right)^{j}}{(\frac{1-z}{h})^{\beta}}.$$
By the uniqueness of the coefficients we have the result.
\item The identity follows from the previous item $(i)$ and \eqref{semigroup-k}. Indeed,
\begin{align*}
\varphi^{h}_{\alpha,\beta+\gamma}(n,j)=&h^{\beta+\gamma}\displaystyle\sum^{j}_{w=0} \binom{j}{w}(-1)^{i}h^{w-\alpha w}k^{\beta+\gamma-\alpha w}(n)\\
=&h^{\beta+\gamma}\displaystyle\sum^{j}_{w=0} \binom{j}{w}(-1)^{i}h^{w-\alpha w} \sum^{n}_{i=0} k^{\beta}(n-i)k^{\gamma-\alpha w}(i)
\\
=&h^{\beta+\gamma}\sum^{n}_{i=0} k^{\beta}(n-i)\displaystyle\sum^{j}_{w=0} \binom{j}{w}(-1)^{w}h^{w-\alpha w} k^{\gamma-\alpha w}(i)
\\
=& h^{\beta} \displaystyle\sum_{i=0}^{n} k^{\beta }(n-i) \varphi^{h}_{\alpha,\gamma}(i,j).
\end{align*}
\item Note that, by \eqref{ScaledWright}
\begin{align*}
\frac{1}{h^{n}\Gamma(n)}\displaystyle\int_0^{\infty}e^{-s/h}s^{n-1}\psi_{\alpha,\beta}(s,t)  \ ds &= \frac{1}{h^{n}\Gamma(n)}\displaystyle\int_0^{\infty}e^{-s/h}s^{n+\beta-2} W_{-\alpha,\beta} (-ts^{-\alpha})\ ds\\
&=\frac{1}{h^{n}}\displaystyle\sum_{i=0}^{\infty}\dfrac{(-t)^{i}}{\Gamma(n)\Gamma(-\alpha i+\beta)i!}\displaystyle\int_0^{\infty} e^{-s/h}s^{n+\beta-2-\alpha i} \ ds \\
&=h^{\beta-1} \displaystyle\sum_{i=0}^{\infty}h^{-\alpha i}\dfrac{\Gamma(n-1+\beta-\alpha i)}{\Gamma(n)\Gamma(-\alpha i+\beta)}\frac{(-t)^{i}}{i!} \\
&=h^{\beta-1} \displaystyle\sum_{i=0}^{\infty}h^{-\alpha i}k^{\beta-\alpha i}(n-1)\frac{(-t)^{i}}{i!}.
\end{align*}
On the other hand, by $(i)$ we obtain
\begin{align*}
\sum_{j=1}^{\infty} \varphi^{h}_{\alpha,\beta}(n-1,j-1)&\frac{t^{j-1}}{h^{j}(j-1)!}=h^{\beta}\sum_{j=1}^{\infty}\displaystyle\sum^{j-1}_{i=0} \binom{j-1}{i}(-1)^{i}h^{i-\alpha i}k^{\beta-\alpha i}(n-1)\frac{t^{j-1}}{h^{j}(j-1)!}\\
&=h^{\beta-1}\sum_{j=0}^{\infty} \displaystyle\sum^{j}_{i=0} \binom{j}{i}(-1)^{i}h^{i-\alpha i}k^{\beta-\alpha i}(n-1)\left(\frac{t}{h}\right)^{j}\frac{1}{j!} \\
&=h^{\beta-1}\sum_{i=0}^{\infty}\sum_{j=i}^{\infty} \binom{j}{i}(-1)^{i}h^{i-\alpha i}k^{\beta-\alpha i}(n-1) \left(\frac{t}{h}\right)^{j}\frac{1}{j!}  \\
&=h^{\beta-1}\sum_{i=0}^{\infty} (-1)^{i}h^{i-\alpha i} k^{\beta-\alpha i}(n-1)   \sum_{j=i}^{\infty} \binom{j}{i}\left(\frac{t}{h}\right)^{j}\frac{1}{j!}  \\
&=h^{\beta-1}\sum_{i=0}^{\infty} (-1)^{i}h^{i-\alpha i}k^{\beta-\alpha i}(n-1)   \sum_{j=i}^{\infty}  \frac{1}{i!(j-i)!} \left(\frac{t}{h}\right)^{j}\\
&=h^{\beta-1}\sum_{i=0}^{\infty} (-1)^{i}h^{i-\alpha i}k^{\beta-\alpha i}(n-1)   \sum_{j=0}^{\infty}  \frac{1}{i!j!} \left(\frac{t}{h}\right)^{j+i}\\
&=h^{\beta-1}\sum_{i=0}^{\infty} (-1)^{i}h^{-\alpha i}k^{\beta-\alpha i}(n-1)   \frac{t^{i} }{i!} e^{t/h},
\end{align*}
for all $n\in\N$. Thus, the result is proved.
\item  Let $h>0$ and  $\lambda\in\C$ such that $\vert \lambda\vert<\dfrac{1}{h^{\alpha}}$. Then,
\begin{align*}
\displaystyle\sum^{\infty}_{j=1}\varphi^{h}_{\alpha,\beta}(n-1,j-1)\frac{1}{h^{j}}\left(\frac{1}{h}-\lambda\right)^{-j}&=\displaystyle\sum^{\infty}_{j=1}\varphi^{h}_{\alpha,\beta}(n-1,j-1)\frac{1}{h^{j}}\int^{\infty}_{0}e^{-\left(\frac{1}{h}-\lambda\right)s}\frac{s^{j-1}}{(j-1)!}ds\\
&= \int^{\infty}_{0}e^{-\left(\frac{1}{h}-\lambda\right)s}\displaystyle\sum^{\infty}_{j=1}\varphi^{h}_{\alpha,\beta}(n-1,j-1)\frac{1}{h^{j}}\frac{s^{j-1}}{(j-1)!}ds \\
&=\frac{1}{h^{n}\Gamma(n)}\int^{\infty}_{0}e^{\lambda s}\displaystyle\int_0^{\infty}e^{-t/h}t^{n-1}\psi_{\alpha,\beta}(t,s) \ dt \ ds \\
&=\frac{1}{h^{n}\Gamma(n)}  \displaystyle\int_0^{\infty}e^{-t/h}t^{n-1} t^{\alpha+\beta-1} E_{\alpha,\alpha+\beta}(\lambda t^{\alpha}) \ dt   \\
&=\frac{1}{h}\mathcal{E}^h_{\alpha,\alpha+\beta}(\lambda,n),
\end{align*}
where we have used item $(iii)$, \cite[Theorem 3 (iii)]{Abadia-Miana} and \cite[Theorem 2.8.]{Alvarez-Diaz-Lizama-2020}.
\item The result is obtained as follows:
\begin{align*}
\varphi^{h}_{\alpha,\beta}(n,j)-&\varphi^{h}_{\alpha,\beta}(n,j+1)\\
&= h^{\beta}\displaystyle\sum^{j}_{i=0} \binom{j}{i}(-1)^{i} h^{i-\alpha i}k^{\beta-\alpha i}(n) - h^{\beta} \displaystyle\sum^{j+1}_{i=0} \binom{j+1}{i}(-1)^{i}h^{i-\alpha i}k^{\beta-\alpha i}(n)\\
&=-h^{\beta}\displaystyle\sum^{j+1}_{i=1} \binom{j}{i-1}(-1)^{i}h^{i-\alpha i}k^{\beta-\alpha i}(n)\\
&=h^{\beta+1-\alpha }\displaystyle\sum^{j}_{i=0} \binom{j}{i}(-1)^{i}h^{i-\alpha i} k^{\beta-\alpha (i+1)}(n) \\
&=h\varphi^{h}_{\alpha,\beta-\alpha}(n,j),
\end{align*}
where we have used \cite[Section 1.4, Eq. (5)]{Aigner-(2006)}.
\item  By item $(i)$, \cite[Section 1.4, Eq. (5)]{Aigner-(2006)} and \eqref{semigroup-k} it follows
\begin{small}
 \begin{align*}
\varphi^h_{\alpha,0}(n,j+1)&=\displaystyle\sum^{j+1}_{i=0} \binom{j+1}{i}(-1)^{i}h^{i-\alpha i}k^{-\alpha i}(n) \\
&=\displaystyle\sum^{j}_{i=0} \binom{j}{i}(-1)^{i}h^{i-\alpha i}k^{-\alpha i}(n)+\displaystyle\sum^{j+1}_{i=1} \binom{j}{i-1}(-1)^{i}h^{i-\alpha i}k^{-\alpha i}(n)\\
&=\displaystyle\sum^{j}_{i=0} \binom{j}{i}(-1)^{i}h^{i-\alpha i}k^{-\alpha i}(n)-\displaystyle\sum^{j}_{i=0} \binom{j}{i}(-1)^{i}h^{i-\alpha i}h^{1-\alpha }k^{-\alpha (i+1)}(n)\\
&=\displaystyle\sum^{n}_{p=0}\displaystyle\sum^{j}_{i=0} \binom{j}{i}(-1)^{i}h^{i-\alpha i}k^{-\alpha i}(n-p)k^{0}(p)\\
&\qquad \qquad -\displaystyle\sum^{n}_{p=0}\displaystyle\sum^{j}_{i=0} \binom{j}{i}(-1)^{i}h^{i-\alpha i}h^{1-\alpha }k^{-\alpha i}(n-p)k^{-\alpha}(p) \\
&=\displaystyle\sum^{n}_{p=0}\displaystyle\sum^{j}_{i=0} \binom{j}{i}(-1)^{i}h^{i-\alpha i}k^{-\alpha i}(n-p)\left(k^{0}(p)-h^{1-\alpha}k^{-\alpha}(p)\right)\\
&=\displaystyle\sum^{n}_{p=0}\displaystyle\sum^{j}_{i=0} \binom{j}{i}(-1)^{i}h^{i-\alpha i}k^{-\alpha i}(n-p) \displaystyle\sum^{1}_{i=0} \binom{1}{i}(-1)^{i}h^{i-\alpha i}k^{-\alpha i}(p) \\
&=\displaystyle\sum^{n}_{p=0} \varphi^h_{\alpha,0}(n-p,j)\varphi^h_{\alpha,0}(p,1).
\end{align*}
\end{small}
\item From Definition \ref{Def-Wright-discreta},  we have that $\varphi^h_{\alpha,0}(n,0)=\delta_0(n)$ for $n\in\N_0.$ Furthermore,
\begin{align*}
\varphi^h_{\alpha,0}(n,1)&=k^{0}(n)-h^{1-\alpha}k^{-\alpha}(n).
\end{align*}
Then,
$$\varphi^h_{\alpha,0}(0,1)=1-h^{1-\alpha}  \geq 0 $$
and
$$\varphi^h_{\alpha,0}(n,1)=h^{1-\alpha}\frac{\alpha(1-\alpha)(2-\alpha)\cdots(n-1-\alpha)}{n!}\geq 0, \ n\in\N.$$
By $(i)$ of the Proposition \ref{prop-Cesaro} and items $(vi)$ and $(ii)$ the result follows.
\item The identity is a particular case of  \eqref{Wright-discreta}, by letting $z\rightarrow 1^{-}$ with $z\in\R$.
\item By \eqref{Wright-discreta} and \eqref{ZT-K}, we have
\begin{align*}
\displaystyle\sum_{l=0}^{\infty}\varphi^h_{\alpha,\beta}(n,l)  k^{\gamma}(j)&= \frac{1}{2 \pi i} \displaystyle\int_{\Upsilon} \frac{1}{z^{n+1}}\displaystyle\sum_{l=0}^{\infty}\frac{\left(1-h(\frac{1-z}{h})^{\alpha} \right)^{j}}{(\frac{1-z}{h})^{\beta}} k^{\gamma}(j)\,dz \\
&=\frac{1}{h^{\gamma}}\frac{1}{2 \pi i} \displaystyle\int_{\Upsilon} \frac{1}{z^{n+1}}\frac{1}{(\frac{1-z}{h})^{\beta+\gamma\alpha}} \,dz \\
&=h^{\beta+\gamma(\alpha-1)}k^{\beta+\gamma\alpha}(n).
\end{align*}
\end{enumerate}
\vspace*{-1cm}
\end{proof}


\begin{remark} Let $0<\alpha<1$.
Taking $\lambda=0$ in  Proposition \ref{Prop-Levy}-$(iv)$, we have
\begin{align}\label{DensityInJ}
\displaystyle\sum^{\infty}_{j=0}\varphi^{h}_{\alpha,1-\alpha}(n-1,j)=1, \qquad  n\in\N.
\end{align}
\end{remark}

\begin{remark}
Some of the results obtained in the previous proposition can be
found in the work carried out by Alvarez et al. in \cite{Alvarez-Diaz-Lizama-2020}  with $h=1$.
\end{remark}

\section{Fundamental solution}\label{homogeneo}
Here we investigate the representation of the solution to the fractional diffusion equation \eqref{1.2}, and we prove several interesting properties related to it.

Let us start recalling the problem. Let $h>0$ and $0<\alpha< 1$. Consider the fractional diffusion equation in discrete time on the Lebesgue $L^p(\R^N)$ spaces, given by
\begin{equation}\label{Main}
\left\{
\begin{array}{lll}
\dC u(nh,x) =\Delta u(nh,x),\quad n\in\N,\,x\in\mathbb{R}^N, \\ \\
u(0,x) =f(x),\\
\end{array}
\right.
\end{equation}
where $u$ and  $f$ are function defined on $\N^{h}_0\times\R^{N}$ and $\R^{N}$ respectively.

Let us define the fundamental solution
\begin{align} \label{subordination1}
\mathcal{G}^{\alpha}_{n,h}(x):=\sum_{j=1}^{\infty} \varphi^{h}_{\alpha,1-\alpha}(n-1,j-1)\mathcal{G}_{j,h}(x),\quad n\in\N,\,x\in\mathbb{R}^N,
\end{align}
where the functions $\mathcal{G}_{n,h}(x)$  denote the  discrete Gaussian  defined by \eqref{discretegaussian}.

The next result shows that $\mathcal{G}^{\alpha}_{n,h}\ast f$ is the solution of \eqref{Main}.\\
\begin{theorem}\label{Theo-Sol} Let $f$  be a function on $L^p(\R^N).$ For $h>0$ and $0<\alpha< 1$, the function
\begin{align}\label{sol}
u(nh,x):=(\mathcal{G}_{n,h}^{\alpha}*f)(x)
\end{align}
is the unique solution of the fractional diffusion  equation  in discrete time \eqref{Main} on the Lebesgue $L^p(\R^N)$ spaces.
\end{theorem}
\begin{proof}
First of all, note that  by Proposition \ref{DG-Properties}-$(ii)$ and \eqref{DensityInJ}, we can conclude that 
\begin{equation}\label{Eq5.4}
\int_{\R^N}\mathcal{G}_{n,h}^{\alpha}(x)\,dx=1.
\end{equation}
Consequently, we have $\|u(nh,x)\|_p\leq \|f\|_{p}.$  

Now, we see that $u$ satisfies \eqref{Main}. Equation \eqref{SolFirtOrderEquation} implies
\begin{small}
\begin{align*}
\Delta u(nh,x)&=\sum_{j=1}^{\infty} \varphi^{h}_{\alpha,1-\alpha}(n-1,j-1)\delta_{\text{left}}(\mathcal{G}_{j,h}\ast f)(x)\\
&=h^{-1}\sum_{j=1}^{\infty} \varphi^{h}_{\alpha,1-\alpha}(n-1,j-1)(\mathcal{G}_{j,h}\ast f)(x)-h^{-1}\sum_{j=1}^{\infty} \varphi^{h}_{\alpha,1-\alpha}(n-1,j-1)(\mathcal{G}_{j-1,h}\ast f)(x)\\
&=h^{-1}\sum_{j=1}^{\infty} \varphi^{h}_{\alpha,1-\alpha}(n-1,j-1)(\mathcal{G}_{j,h}\ast f)(x)-h^{-1}\sum_{j=0}^{\infty} \varphi^{h}_{\alpha,1-\alpha}(n-1,j)(\mathcal{G}_{j,h}\ast f)(x)\\
&=h^{-1}\sum_{j=1}^{\infty} \varphi^{h}_{\alpha,1-\alpha}(n-1,j-1)(\mathcal{G}_{j,h}\ast f)(x)-h^{-1}\sum_{j=1}^{\infty} \varphi^{h}_{\alpha,1-\alpha}(n-1,j)(\mathcal{G}_{j,h}\ast f)(x)\\
&\qquad\qquad-h^{-1}\varphi^{h}_{\alpha,1-\alpha}(n-1,0)f(x).
\end{align*}
Now, by Proposition \ref{Prop-Levy}-$(ii)$ and \eqref{kernel_k_1}, we have
\begin{align*}
h^{-1}\varphi^{h}_{\alpha,1-\alpha}(n-1,0)f(x)&=h^{-1}h^{1-\alpha}\sum^{n-1}_{i=0}k^{1-\alpha}(n-1-i)\varphi^{h}_{\alpha,0}(i,0)f(x)\\
&=h^{-\alpha}\sum^{n-1}_{i=0}k^{1-\alpha}(n-1-i)k^{0}(i)f(x)\\
&=h^{-\alpha} k^{1-\alpha}(n-1)f(x).
\end{align*}
Then,  by $(v)$ of Proposition \ref{Prop-Levy}, we get
\begin{align*}
\Delta u(nh,x)&=h^{-1}\sum_{j=1}^{\infty} \varphi^{h}_{\alpha,1-\alpha}(n-1,j-1)(\mathcal{G}_{j,h}\ast f)(x)-h^{-1}\sum_{j=1}^{\infty} \varphi^{h}_{\alpha,1-\alpha}(n-1,j)(\mathcal{G}_{j,h}\ast f)(x)\\
&\qquad\qquad-h^{-\alpha} k^{1-\alpha}(n-1)f(x)\\
&=h^{-1}\sum_{j=1}^{\infty}\left( \varphi^{h}_{\alpha,1-\alpha}(n-1,j-1)-\varphi^{h}_{\alpha,1-\alpha}(n-1,j)\right)(\mathcal{G}_{j,h}\ast f)(x)-h^{-\alpha} k^{1-\alpha}(n-1)f(x)\\
&=h^{-1}\sum_{j=0}^{\infty}\left( \varphi^{h}_{\alpha,1-\alpha}(n-1,j)-\varphi^{h}_{\alpha,1-\alpha}(n-1,j+1)\right)(\mathcal{G}_{j+1,h}\ast f)(x)-h^{-\alpha} k^{1-\alpha}(n-1)f(x)\\
&=\sum_{j=0}^{\infty} \varphi^{h}_{\alpha,1-2\alpha}(n-1,j) (\mathcal{G}_{j+1,h}\ast f)(x)-h^{-\alpha} k^{1-\alpha}(n-1)f(x).
\end{align*}
By the previous identity and $(ii)$ of Proposition \ref{Prop-Levy}, we have that
\begin{align*}
h^{\alpha}\Delta \sum^{n}_{w=1} &k^\alpha(n-w)  u(wh,x)=h^{\alpha}\Delta \sum^{n-1}_{w=0} k^\alpha(n-1-w)  u((w +1)h,x)\\
&=h^{\alpha}\sum^{n-1}_{w=0} k^\alpha(n-1-w)\sum_{j=0}^{\infty} \varphi^{h}_{\alpha,1-2\alpha}(w,j) (\mathcal{G}_{j+1,h}\ast f)(x)\\
&\qquad - \sum^{n-1}_{w=0} k^\alpha(n-1-w) k^{1-\alpha}(w)f(x) \\
&=h^{1-\alpha}\sum^{n-1}_{w=0} k^\alpha(n-1-w)\sum_{j=0}^{\infty} \sum_{p=0}^{w} k^{1-2\alpha}(w-p)\varphi^{h}_{\alpha,0}(p,j) (\mathcal{G}_{j+1,h}\ast f)(x) -f(x) \\
&=h^{1-\alpha}\sum_{j=0}^{\infty} \sum^{n-1}_{w=0} k^\alpha(n-1-w)\sum_{p=0}^{w} k^{1-2\alpha}(w-p)\varphi^{h}_{\alpha,0}(p,j) (\mathcal{G}_{j+1,h}\ast f)(x)-f(x) \\
&=h^{1-\alpha}\sum_{j=0}^{\infty}  \sum_{p=0}^{n-1} k^{1-\alpha}(n-1-p)\varphi^{h}_{\alpha,0}(p,j) (\mathcal{G}_{j+1,h}\ast f)(x) - f(x) \\
&= \sum_{j=0}^{\infty} \varphi^{h}_{\alpha,1-\alpha}(n-1,j) (\mathcal{G}_{j+1,h}\ast f)(x)  -f(x) \\
&=u(nh,x)- f(x),
\end{align*}
that is,
\begin{align*}
u(nh,x)= h^{\alpha} \Delta \sum^{n}_{w=1} k^\alpha(n-w)u(wh,x) +f(x).
\end{align*}
Now, convolving the above identity by $k^{1-\alpha}$ and multiplying by $h^{-\alpha}$, we obtain
\begin{align*}
h^{-\alpha}\sum_{j=0}^{n}k^{1-\alpha}(n-j) u(jh,x)&=\Delta \sum_{j=0}^{n} u(j h,x)-\Delta u(0,x)+h^{-\alpha}k^{2-\alpha}(n) f(x)\\
&=\Delta \sum_{j=1}^{n} u(j h,x) +h^{-\alpha}k^{2-\alpha}(n) f(x).
\end{align*}
By Proposition \ref{delta-k}, we can conclude that
\begin{align*}
h^{-\alpha}\sum_{k=0}^{n}k^{1-\alpha}(n-j) u(jh,x)-&h^{-\alpha}\sum_{k=0}^{n-1}k^{1-\alpha}(n-1-j) u(jh,x)\\
& \qquad = \Delta  u(nh,x)+h^{-\alpha}k^{1-\alpha}(n) f(x).
\end{align*}
Hence, the result follows from Proposition \ref{relation-CR}.
\end{small}
\end{proof}

In the following results we show other representations for $\mathcal{G}^{\alpha}_{n,h}$.  In the first result, we represent $\mathcal{G}_{n,h}^{\alpha}(x)$ using the Poisson transform of the Gaussian kernel while in the second one we use the Fox H-function. This fact in turn gives other representations of the solution \eqref{sol}. \\

\begin{proposition}\label{GaussianSemigroup-FS} Let $0<h$  and $0<\alpha< 1$. Then, \eqref{subordination1} is equivalent to
\begin{align}\label{psi-1-alfha}
\mathcal{G}^{\alpha}_{n,h}(x)=\frac{1}{h^{n}}\displaystyle\int_{0}^{\infty}\displaystyle\int_{0}^{\infty}  e^{-s/h}\frac{s^{n-1}}{(n-1)!}\psi_{\alpha,1-\alpha}(s,t)G_{t}(x)\ ds \ dt, \quad n\in\N, \ x\in\R^N \setminus \{ 0\},
\end{align}
where  $G_{t}$ is the Gaussian kernel \eqref{gaussian-semigroup} and $\psi_{\alpha,\beta}$  is \eqref{ScaledWright}.
\end{proposition}
\begin{proof}
From Proposition \ref{Prop-Levy} part $(iii)$, we get
\begin{align*}
\mathcal{G}^{\alpha}_{n,h}(x)&= \sum_{j=1}^{\infty} \varphi^{h}_{\alpha,1-\alpha}(n-1,j-1)\mathcal{G}_{j,h}(x)\\
&=\sum_{j=1}^{\infty} \varphi^{h}_{\alpha,1-\alpha}(n-1,j-1)\frac{1}{h^{j}}\displaystyle\int_{0}^{\infty}e^{-t/h}\frac{t^{j-1}}{(j-1)!}G_{t}(x)\ dt \\
&=\displaystyle\int_{0}^{\infty}e^{-t/h}\sum_{j=1}^{\infty}\varphi^{h}_{\alpha,1-\alpha}(n-1,j-1)\frac{1}{h^{j}}\frac{t^{j-1}}{ (j-1)!}G_{t}(x)\ dt \\
&=\frac{1}{h^{n}}\displaystyle\int_{0}^{\infty}\displaystyle\int_{0}^{\infty}  e^{-s/h}\frac{s^{n-1}}{(n-1)!}\psi_{\alpha,1-\alpha}(s,t)G_{t}(x)\ ds \ dt.
\end{align*}
The result follows.
\end{proof}

\begin{proposition}\label{FoxH-FS} Let $0<h$ and $0<\alpha< 1$.  Then
\begin{align*}
\mathcal{G}^{\alpha}_{n,h}(x)= \frac{1}{ \Gamma(n)\pi^{N/2}\vert x \vert^{N}}  H^{30}_{13}\left[\begin{array}{c}
 \dfrac{\vert x\vert^{2}}{4h^{\alpha}} \left\vert \begin{array}{c}
(1,\alpha)  \\
(n,\alpha),(\frac{N}{2},1), (1,1)
\end{array}
\right.
\end{array}
\right],
\end{align*}
where $H^{03}_{31}$ denotes the Fox H-function.
\end{proposition}
\begin{proof}
By \cite[Theorem 3.1]{Bazhlthesis}, \cite[Theorem 15-$(ii)$.]{Abadia-Miana} and \cite[Theorem 2.12]{Kemppainen}, we have the following subordination formula
\begin{align*}
\frac{1}{\pi^{N/2}\vert x \vert^{N}}H^{02}_{21}\left[\begin{array}{c}
 \dfrac{4t^{\alpha}}{\vert x\vert^{2}} \left\vert \begin{array}{c}
(1-\frac{N}{2},1),(0,1)\\
(0,\alpha)
\end{array}
\right.
\end{array}
\right] &=\displaystyle\int_{0}^{\infty} \psi_{\alpha,1-\alpha}(t,s)G_{s}(x)\ ds.
\end{align*}
Now,
\begin{align*}
\mathcal{G}^{\alpha}_{n,h}(x)&=\frac{1}{h^{n}\Gamma(n)\pi^{N/2}\vert x \vert^{N}}\displaystyle\int_{0}^{\infty} e^{-t/h}t^{n-1} H^{12}_{32}\left[\begin{array}{c}
 \dfrac{4 t^{\alpha}}{\vert x\vert^{2}} \left\vert \begin{array}{c}
(1-\frac{N}{2},1),\ (0,1),\ (0,1)\\
(0,1), \ (0,\alpha)
\end{array}
\right.
\end{array}
\right] \, dt  \\
&=\frac{1}{ \Gamma(n)\pi^{N/2}\vert x \vert^{N}}  H^{13}_{42}\left[\begin{array}{c}
 \dfrac{4 h^{\alpha}}{\vert x\vert^{2}} \left\vert \begin{array}{c}
(1-n,\alpha),\ (1-\frac{N}{2},1),\  (0,1), \ (0,1)\\
(0,1), \ (0,\alpha)
\end{array}
\right.
\end{array}
\right] \\
&=\frac{1}{ \Gamma(n)\pi^{N/2}\vert x \vert^{N}}  H^{03}_{31}\left[\begin{array}{c}
 \dfrac{4h^{\alpha}}{\vert x\vert^{2}} \left\vert \begin{array}{c}
(1-n,\alpha),(1-\frac{N}{2},1), (0,1)\\
(0,\alpha)
\end{array}
\right.
\end{array}
\right]\\
&= \frac{1}{ \Gamma(n)\pi^{N/2}\vert x \vert^{N}}  H^{30}_{13}\left[\begin{array}{c}
 \dfrac{\vert x\vert^{2}}{4h^{\alpha}} \left\vert \begin{array}{c}
(1,\alpha)  \\
(n,\alpha),(\frac{N}{2},1), (1,1)
\end{array}
\right.
\end{array}
\right],
\end{align*}
where we have used  Corollary 2.3.1, Proposition 2.2  and  Proposition 2.3 of \cite{H-transforms}.
\end{proof}

The following proposition states some basic properties of the fundamental solution.\\

\begin{proposition}\label{HeatProperties}
The function $\mathcal{G}^{\alpha}_{n,h}$  satisfies:
\begin{enumerate}[$(i)$]
\item $\displaystyle\mathcal{G}_{n,h}^{\alpha}(x)>0,\quad n\in\N, \ 0<h\leq 1.$
\item $\displaystyle \int_{\R^N}\mathcal{G}_{n,h}^{\alpha}(x)\,dx=1$.
\item $\displaystyle\mathcal{F}({\mathcal{G}}_{n,h}^{\alpha})(\xi)=\frac{1}{h}\mathcal{E}^h_{\alpha,1}(-\vert\xi \vert^2,n),$ \quad $\xi\in\R^N.$
\item $\displaystyle \int_{\R^N}|x|^2\mathcal{G}_{n,h}^{\alpha}(x)\,dx=\Gamma(3)Nh^{\alpha}k^{\alpha+1}(n-1).$
\end{enumerate}
\end{proposition}
\begin{proof}
$(i)$ Follows from $(vi)$ of Proposition \ref{Prop-Levy} and $(i)$ of the Proposition \ref{DG-Properties}. 
$(ii)$ was showed in the proof of Theorem \ref{Theo-Sol} (see \eqref{Eq5.4}).
Next, let us prove $(iii)$. Since $\mathcal{F}({G}_{t})(\xi)=e^{-t|\xi|^2}$, for $\xi\in\R^N$ by Proposition \ref{gauss-prop} part $(iii)$, it follows from \cite[Theorem 3]{Abadia-Miana} that
\begin{align*}
\int_0^{\infty} \psi_{\alpha,1-\alpha}(s,t)e^{-t\vert \xi \vert^{2}}\, dt=E_{\alpha,1}(-\vert \xi\vert^2 s^{\alpha}).
\end{align*}
Equation \eqref{MT-discrete} implies that
$$\mathcal{F}(\mathcal{G}^{\alpha}_{n,h})(\xi)=\frac{1}{h^n}\int_{0}^{\infty}  e^{-s/h}\frac{s^{n-1}}{\Gamma(n)}E_{\alpha,1}(-\vert \xi\vert^2 s^{\alpha})\,ds =\frac{1}{h}\mathcal{E}^h_{\alpha}(-\vert\xi \vert^2,n).$$
Finally, by Fubini's Theorem, Proposition \ref{gauss-prop} and \cite[Theorem 3]{Abadia-Miana}, we have that
\begin{align*}
\displaystyle \int_{\R^N}|x|^2\mathcal{G}_{n,h}^{\alpha}(x)\,dx=&2N\frac{1}{h^{n}}\displaystyle\int_{0}^{\infty}\displaystyle\int_{0}^{\infty}  e^{-s/h}\frac{s^{n-1}}{(n-1)!}\psi_{\alpha,1-\alpha}(s,t)t\, dt\,ds\\
=&2N\Gamma(2)\frac{1}{h^{n}}\displaystyle\int_{0}^{\infty}  e^{-s/h}\frac{s^{n-1}}{(n-1)!} g_{\alpha+1}(s) \,ds \\
=&\Gamma(3)Nh^{\alpha}k^{\alpha+1}(n-1).
\end{align*}
Thus, we get item $(iv)$.
\end{proof}


\begin{remark}We recall that  the total mass and first moment of  the function  $$w(nh,x)=(\mathcal{G}_{n,h}\ast f)(x)$$
are conservative (see \cite[Remark 2.6]{Abadia-Alvarez-20}). Then, we have that the total mass of solution of \eqref{Main}, given by
$$u(nh,x)=\sum_{j=1}^{\infty} \varphi^{h}_{\alpha,1-\alpha}(n-1,j-1)(\mathcal{G}_{j,h}\ast f)(x)$$
is conservative. Indeed,
\begin{align*}
\int_{\R^N}u(nh,x)\,dx=\sum_{j=1}^{\infty} \varphi^{h}_{\alpha,1-\alpha}(n-1,j-1)\int_{\R^N}(\mathcal{G}_{j,h}\ast f)(x) \,dx=\int_{\R^N} f(x)\,dx,
\end{align*}
where in the last equality we have used \eqref{DensityInJ}.
The first moment is also conservative:
$$\int_{\R^N}x\,u(nh,x)dx=\sum_{j=1}^{\infty} \varphi^{h}_{\alpha,1-\alpha}(n-1,j-1)\int_{\R^N}x\,(\mathcal{G}_{j,h}\ast f)(x) \,dx=\int_{\R^N}xf (x)\,dx,$$
as long as $(1+|x|)f\in L^{1}(\R^N)$.
However, in the same way that $w$, the second moment of $u$ is not conserved in time.
\end{remark}
\section{Asymptotic decay and large-time behavior of solutions for the fractional diffusion equation in discrete time}
Now we will present the asymptotic decay of the solution of \eqref{Main} (which is given by \eqref{sol}) in $L^p$ spaces and the corresponding large-time behaviour.
\subsection{Asymptotic decay}\label{4}
In this part we show the following estimates of the fundamental solution $\mathcal{G}_{n,h}^{\alpha}$ in $L^p$-spaces. Finally we also state $L^p$-estimates for $\nabla\mathcal{G}_{n,h}^{\alpha},$ which are useful for study the large time behaviour of solutions of \eqref{Main} in Lebesgue spaces.\\

\begin{lemma}\label{G-dacay} Let $0<\alpha<1.$ Then $$
\|\mathcal{G}_{n,h}^{\alpha} \|_{p} \leq C_p\dfrac{1}{(nh)^{\frac{\alpha N}{2}(1-1/p)}},\quad n\in\N,
$$
 for $p\in [1,\infty]$ if $N=1,$ for $p\in[1,\infty)$ if $N=2,$ and for $p\in[1,\frac{N}{N-2})$ if $N>2.$
\end{lemma}

\begin{proof}
It is well known (see \cite[p.334 (3.326)]{G}) that there exists $C_p$ (independent of $t$) such that $||G_t||_{p} = C_p\frac{1}{t^{\frac{N}{2}(1-\frac{1}{p})}}.$  Then for $n$ large enough and the values of $p$ given in the hypothesis, by \eqref{psi-1-alfha} and \cite[Theorem 3 (vi)]{Abadia-Miana} one gets
$$
\|\mathcal{G}^{\alpha}_{n,h}\|_{p}\leq \frac{C_p}{h^n\Gamma(n)}\int_{0}^{\infty}e^{\frac{-t}{h}} t^{n-\alpha\frac{N}{2}(1-\frac{1}{p})-1}\,dt=C_p\frac{\Gamma(n-\alpha\frac{N}{2}(1-\frac{1}{p}))}{h^{\alpha\frac{N}{2}(1-\frac{1}{p})}\Gamma(n)}\leq \frac{C_p}{(nh)^{\alpha\frac{N}{2}(1-\frac{1}{p})}},
$$
where we have applied the asymptotic behaviour of the Gamma function (see \cite{ET}). Since the function $\mathcal{G}^{\alpha}_{n,h}$ belongs to $L^p(\R^N)$ for all $n\in\N,$ then the result is valid for all $n\in\N.$
\end{proof}

Next, let us present a result about the  $L^p-L^q$ asymptotic decay for $u.$\\
\begin{theorem}\label{teorema4.1} Let $1\leq q\leq p\leq \infty.$ If $f\in L^q(\R^N)$, then the solution $u$ of \eqref{Main} satisfies

\begin{itemize}
\item[$(i)$]  If $q=\infty$ then $
\| u(nh) \|_{\infty} \leq\| f\|_{\infty}.
$
\item[$(ii)$]  If $1\leq q<\infty$ and $N>2q$, then for each $p\in[q,\frac{Nq}{N-2q})$ \begin{equation}\label{eqdecay}
\| u(nh) \|_{p} \leq C_p\dfrac{1}{(nh)^{\frac{\alpha N}{2}(1/q-1/p)}}\| f\|_{q}.
\end{equation}
\item[$(iii)$]  If $1\leq q<\infty$ and $N=2q$, then for each $p\in[q,\infty)$ the estimate \eqref{eqdecay} holds.
\item[$(iv)$]  If $1\leq q<\infty$ and $N<2q$, then for each $p\in[q,\infty]$ the estimate \eqref{eqdecay} holds.
\end{itemize}
Here, $C_p$ is a constant independent of $h$ and $n$.
\end{theorem}

\begin{proof}
Take $r\geq 1$ such that $1+1/p=1/q+1/r,$ and applying Young's inequality we get $$\| u(nh) \|_{p} =\|  \mathcal{G}_{n,h}^{\alpha}*f\|_{p} \leq \|  \mathcal{G}_{n,h}^{\alpha}\|_{r}\| f\|_{q}.$$ Now, we apply Lemma \ref{G-dacay} to estimate $\|  \mathcal{G}_{n,h}^{\alpha}\|_{r}.$ For the case $(i)$, if $q=\infty,$ then $p=\infty,r=1,$ and therefore since $\|\mathcal{G}_{n,h}^{\alpha}\|_1=1,$ the result follows. Note that in the case $(ii)$, if $1\leq q<\infty$ and $N>2q,$ then the condition $q\leq p<\frac{Nq}{N-2q}$ implies $1\leq r <\frac{N}{N-2}.$ So, by Lemma \ref{G-dacay} we get the desired estimates. The cases $(iii)$ and $(iv)$ follow in a similar way.
\end{proof}

\begin{lemma}\label{Nabla-G-Estimates} Let $0<\alpha<1.$ Then $$
\|\nabla\mathcal{G}_{n,h}^{\alpha} \|_{p} \leq C_p\dfrac{1}{(nh)^{\frac{\alpha N}{2}(1-1/p)+\frac{\alpha}{2}}},\quad n\in\N,
$$
 for $p\in [1,\infty)$ if $N=1,$ and for $p\in[1,\frac{N}{N-1})$ if $N>1.$
\end{lemma}
\begin{proof}
The proof is similar to the proof of Lemma \ref{G-dacay} by use of  $\| \nabla G_t \|_{p} =  C_p\frac{1}{t^{\frac{N}{2}(1-\frac{1}{p})+1/2}}$ (see \cite[p.334 (3.326)]{G}).
\end{proof}

\subsection{Large-time behaviour of solutions}\label{5}

In this part we study the asymptotic behaviour of solution $u$ of problem given by \eqref{Main}.
Set
$$M:=\int_{\R^N} f(x)\,dx.$$

Before to show the main result of this section, we  need the following decomposition lemma (see \cite{DZ}).\\
\begin{lemma}\label{decolemma}
Suppose $f\in L^1(\mathbb{R}^N)$ such that $\int_{\mathbb{R}^N}|x||f(x)|dx<\infty.$  Then there exists $F\in L^1(\mathbb{R}^N;\mathbb{R}^N)$ such that
\begin{equation*}
f=\left(\int_{\mathbb{R}^N}f(x)dx\right)\delta_0+\mbox{div}\, F
\end{equation*}
in the distributional sense and
\begin{equation*}
\|F\|_{L^1(\mathbb{R}^N;\mathbb{R}^N)}\leq C_d\int_{\mathbb{R}^N}|x||f(x)|dx.
\end{equation*}
\end{lemma}

\begin{theorem}\label{Theorem5.1}
Let $1\leq p\leq \infty$ and $u$ be the solution of \eqref{Main}.
\begin{itemize}
\item[$(i)$] Then $$(nh)^{\frac{\alpha N}{2}\left(1-\frac{1}{p}\right)}\|u(nh)-M\mathcal{G}^{\alpha}_{n,h}\|_{p}\to 0,\quad \mbox{as}\quad n\to\infty,$$ for $p\in [1,\infty)$ if $N=1,$ and for $p\in[1,\frac{N}{N-1})$ if $N>1,$
\item[$(ii)$] Suppose in addition that $|x|f\in L^1(\R),$ then $$ (nh)^{\frac{\alpha N}{2}(1-\frac{1}{p})}\|u(nh)-M\mathcal{G}^{\alpha}_{n,h}\|_{p}\lesssim (nh)^{-\alpha/2},$$ for $p\in [1,\infty)$ if $N=1,$ and for $p\in[1,\frac{N}{N-1})$ if $N>1.$
\end{itemize}
\end{theorem}
\begin{proof}
First we prove assertion $(ii)$. Since that $f, |x|f\in L^1(\R^N),$ by  decomposition Lemma \ref{decolemma} there exists $\psi\in L^1(\mathbb{R}^N;\mathbb{R}^N)$ such that
\begin{align*}
u(nh,x)&=(\mathcal{G}^{\alpha}_{n,h}\ast (M\delta_0+\mbox{div}\, \psi(\cdot)))(x)\\
&=M_c\mathcal{G}^{\alpha}_{n,h}(x)+(\nabla\mathcal{G}^{\alpha}_{n,h}\ast \psi)(x),
\end{align*}
in the distributional sense, and
$$\|\psi\|_{1}\leq C\||x| f(x)\|_1<\infty.$$
Lemma \ref{Nabla-G-Estimates} implies that
\begin{equation}\label{Part-2-estimates}
\|u(nh)-M\mathcal{G}^{\alpha}_{n,h}\|_{p}\leq C\|\nabla\mathcal{G}^{\alpha}_{n,h}\|_{p}\|xf(x)\|_{1}
\leq C_{p,f}\dfrac{1}{(nh)^{\frac{\alpha N}{2}(1-1/p)+\frac{ \alpha}{2}}},
\end{equation}
Hence the assertion $(ii)$ is proved.

To prove  $(i)$, we choose a sequence $(\eta_j)\subset C_0^{\infty}(\mathbb{R}^N)$ such that  $\int_{\mathbb{R}^N}\eta_j(x)\,dx=M$ for all $j,$ and $\eta_j\to f$ in $L^1(\mathbb{R}^N)$ . For each $j$,  by Lemma  \ref{G-dacay} and \eqref{Part-2-estimates},  we get
\begin{align*}
\|u(nh)-M\mathcal{G}^{\alpha}_{n,h}\|_{p}&\leq \|\mathcal{G}^{\alpha}_{n,h}\ast(f-\eta_j)\|_{p}+\|\mathcal{G}^{\alpha}_{n,h}\ast \eta_j-M_c\mathcal{G}^{\alpha}_{n,h}\|_{p}\\
&\leq \|\mathcal{G}^{\alpha}_{n,h}\|_{p}\|f-\eta_j\|_{1}+\|\mathcal{G}^{\alpha}_{n,h}\ast \eta_j-M_c\mathcal{G}^{\alpha}_{n,h}\|_{p}\\
&\leq C_p\dfrac{1}{(nh)^{\frac{\alpha  N}{2}(1-1/p)}}\|f-\eta_j\|_{1}+C_{p,\eta_j}\dfrac{1}{(nh)^{\frac{\alpha N}{2}(1-1/p)+\frac{\alpha}{2}}}.
\end{align*}
Then
$$\limsup_{n\to\infty}\,(nh)^{\frac{\alpha N}{2}(1-1/p)}
\|u(nh)-M\mathcal{G}^{\alpha}_{n,h}\|_{p}\leq
C_p\|f-\eta_j\|_{1}.$$
The assertion follows by letting $j\to\infty$.
\end{proof}

\end{document}